\documentclass[12pt]{amsart} 

\usepackage[utf8]{inputenc}
\usepackage{amsmath}
\usepackage{amsthm}
\usepackage{amsfonts}
\usepackage{amssymb}
\usepackage{tabu}
\usepackage{latexsym} 
\usepackage[margin=1in]{geometry}
\usepackage{enumerate} 
\usepackage[colorlinks=true]{hyperref}
\usepackage{mathrsfs}
\usepackage{tikz}
\usepackage[parfill]{parskip}
\usetikzlibrary{matrix,arrows,decorations.pathmorphing}
\usepackage{tikz-cd}
\usepackage[all]{xy}
\usepackage{titlesec} 
\usepackage{setspace}
\usepackage{tabularx}
\usepackage{xcolor}
\usepackage{framed}
\usepackage{caption}

\newcommand{\Ext}{\mathrm{Ext}}

\newcommand{\Hom}{\mathrm{Hom}}
\newcommand{\End}{\mathrm{End}}
\newcommand{\sHom}{\underline{\mathrm{Hom}}}
\newcommand{\sEnd}{\underline{\mathrm{End}}}

\newcommand{\ann}{\mathrm{ann}}
\newcommand{\sann}{\underline{\mathrm{ann}}}

\newcommand\ca{\mathrm{ca}}

\newcommand{\m}{\mathfrak{m}}

\newcommand\TT{\mathcal{T}}

\newcommand{\XX}{\mathcal{X}}
\newcommand{\module}{\mathrm{mod}}

\newcommand{\MCM}{\mathrm{MCM}}
\newcommand{\ZCM}{\mathrm{ZCM}}
\newcommand{\sMCM}{\underline{\mathrm{MCM}}}

\newcommand{\add}{\mathrm{add}}

\newcommand{\gldim}{\mathrm{gldim }}

\newcommand{\alex}{\mathcal{A}}
\newcommand{\kalex}{\mathcal{KA}}
\newcommand{\MM}{\mathfrak{CM}}

\newcommand{\coker}{\mathrm{coker}}

\DeclareMathOperator{\cl}{\mathrm{cl}}

\newtheorem{theorem}{Theorem}[section]
\newtheorem*{theorem*}{Theorem}
\newtheorem{lemma}[theorem]{Lemma}
\newtheorem*{lemma*}{Lemma}
\newtheorem{proposition}[theorem]{Proposition}
\newtheorem*{proposition*}{Proposition}
\newtheorem{corollary}[theorem]{Corollary}
\newtheorem*{corollary*}{Corollary}

\theoremstyle{definition}
\newtheorem{definition}[theorem]{Definition}
\newtheorem{example}[theorem]{Example}
\newtheorem{remark}[theorem]{Remark}

\newtheorem{chunk}[theorem]{}

\titleformat{\chapter}[block]{\Huge\scshape\bf\centering}{\thechapter.}{1em}{} 
\titleformat{\section}[block]{\large\scshape\bf\centering}{\thesection.}{1em}{} 
\titleformat{\subsection}[block]{\scshape\bf\centering}{\thesubsection.}{1em}{} 

\author[M.~Akdenizli]{Mert Akdenizli}
\address{Department of Mathematics, Bogazici University, Istanbul, Turkey}
\email{mert.akdenizli@boun.edu.tr}

\author[B.~Aytekin]{Bilal Aytekin}
\address{Department of Mathematics, Bogazici University, Istanbul, Turkey}
\email{bilal.aytekin@boun.edu.tr}

\author[B.~Cetin]{Baran {\c C}etin}
\address{Department of Mathematics, Bilgi University, Istanbul, Turkey}
\email{baran.cetin@bilgiedu.net}

\author[\"{O}.~Esentepe]{\"{O}zg\"{u}r Esentepe}
\address{Department of Mathematics, University of Connecticut, Storrs, CT 06269}
\email{ozgur.esentepe@uconn.edu}

\title{An Alexandrov Topology for Maximal Cohen-Macaulay Modules}


\begin{document}
\maketitle

\begin{abstract}
Using the theory of cohomology annihilators, we define a family of topologies on the set of isomorphism classes of maximal Cohen-Macaulay modules over a Gorenstein ring. We study compactness of these topologies.
\end{abstract}

\section{Introduction}

Recently in \cite{HT}, Hiramatsu and Takahashi introduced a topology on the set of isomorphism classes of maximal Cohen-Macaulay modules over a Cohen-Macaulay local ring. Their method utilizes the concept of degeneration to define a closure operator with the purpose of studying the set $\mathrm{E}(d)$ of isomorphism classes of maximal Cohen-Macaulay modules of multiplicity $d$ as a substitute for the module variety $\mathrm{M}(d)$ of $d$-dimensional modules over a finite dimensional algebra.

In the present paper, we introduce a new family of topologies on the set of isomorphism classes of maximal Cohen-Macaulay modules over a Gorenstein ring using the theory of cohomology annihilators. The main idea is as follows. We first consider the set of all ideals which appear as the annihilator of the stable endomorphism ring of a maximal Cohen-Macaulay module. Then, we put a preorder on this set and consider a topology associated with this preorder. 

Our purpose is to try to better understand the support of Tate cohomology over Gorenstein rings. Let us be more precise and explain our motivation for this project. Hilbert's Syzygy Theorem says that the $d$th syzygy module of a finitely generated module over a polynomial ring $S$ in $d$ variables is a free module \cite{hilbert}. This can be stated, using Ext-modules, in fancier terms as $\Ext_S^{d+1}(M,N) = 0$ for any finitely generated $S$-modules $M$ and $N$. This version also holds true for any ring of global dimension at most $d$. When a ring has infinite global dimension on the other hand, we see that for every $i > 0$, there is a pair of finitely generated modules $M,N$ such that $\Ext^i(M,N)$ is nonzero. Hence, it makes sense to study the annihilators of these nonzero Ext-modules. Particularly, for a commutative Noetherian ring $R$, it is useful to study the \textit{cohomology annihilator ideal}
\begin{align*}
    \ca(R) = \{ r \in R \colon r \Ext^i_R(M,N) = 0 \text{ for all } M,N \in \module R \text{ and } i \gg 0\}
\end{align*}
consisting of uniform annihilators of all Ext-modules of finitely generated modules. When the ring is a Gorenstein local ring, this ideal coincides with the set of those ring elements which annihilate the stable endomorphism ring of every maximal Cohen-Macaulay module or equivalently, which annihilate all Hom-sets in the singularity category.

In general, it is not easy to describe the cohomology annihilator ideal with only a few classes of rings where a complete description is possible. One of these classes is the class of one dimensional reduced complete Gorenstein local rings \cite{Esentepe}. In this case, the cohomology annihilator ideal coincides with the conductor ideal. Given a commutative ring $R$, let $\overline{R}$ denote the integral closure of $R$ in its total ring of fractions. Then, the \textit{conductor ideal} is defined as the annihilator of the $R$-module $\overline{R}/R$. When $R$ is an analytically unramified Gorenstein ring, the conductor is isomorphic to the annihilator of the stable endomorphism ring of $\overline{R}$ as an $R$-module and in dimension one, $\overline{R}$ is maximal Cohen-Macaulay as an $R$-module. Hence, the main result of \cite{Esentepe} can be rephrased as follows: When $R$ is a one dimensional reduced complete Gorenstein local ring, if a ring element $r \in R$ annihilates the stable endomorphism ring of $\overline{R}$, then it annihilates the stable endomorphism ring of every maximal Cohen-Macaulay $R$-module. 

This observation motivates the following question: given a Gorenstein local ring $R$, is it always possible to find a single maximal Cohen-Macaulay module $M$ such that any annihilator of the stable endomorphism ring of $M$ is also an annihilator of the entire singularity category? Towards understanding this question, it is natural to consider the following full subcategory of $\MCM(R)$- the category of maximal Cohen-Macaulay $R$-modules:
\begin{align*}
    \cl(M) = \{ L \in \MCM(R) \colon \sann_R(L) \subseteq \sann_R(M) \}
\end{align*}
where $\sann_R(*) = \ann_R \sEnd_R(*)$ is the annihilator of the stable endomorphism ring of $*$ which we call the \textit{stable annihilator} of $*$. By abuse of notation, we will also consider this as a subcategory of the stable category of maximal Cohen-Macaulay modules.

Further abusing the notation, we consider $\cl(M)$ as a subset of the set of isomorphism classes of maximal Cohen-Macaulay $R$-modules. We define a preorder by declaring $L \leq M$ if and only if $L \in \cl(M)$. In Section 2, we recall definitions from the literature regarding preorders and related topologies. Our motivating question then becomes: is there a single maximal Cohen-Macaulay module $M$ which is in the intersection of all nonempty closed sets? The following theorem is our first main theorem which describes our motivating question in terms of topological properties.

\begin{theorem*}
The set of isomorphism classes of maximal Cohen-Macaulay modules is compact if and only if there is a single maximal Cohen-Macaulay module $M$ such that $\sann_R(M) \subseteq \sann_R(X)$ for every maximal Cohen-Macaulay module $X$.
\end{theorem*}

Therefore, for instance, the set of isomorphism classes of maximal Cohen-Macaulay modules over a one dimensional reduced complete Gorenstein local rings is compact. See Theorem \ref{example1}, Corollary \ref{Rouquier-corollary}, Corollary \ref{example3}.

In Section 3, we study continuous maps between these topological spaces. Continuous maps, in this case, correspond to monotonous (order-preserving) maps. The main result in this section is the following theorem.

\begin{theorem*}
Let $R$ be a Gorenstein local ring and $x$ be a cohomology annihilator which is a nonzerodivisor. Then, the projection map $\pi: R \to R/xR$ induces a continuous map from the set of isomorphism classes of maximal Cohen-Macaulay $R$-modules to the set of isomorphism classes of maximal Cohen-Macaulay $R/xR$-modules.
\end{theorem*}

In Section 4, we study variations of our topology whose closed sets are generated by sets of the form
\begin{align*}
    \cl_n(M) = \left\lbrace L \in \MCM(R) \colon [\sann_R(L)]^n \subseteq \sann_R(M)  \right\rbrace
\end{align*}
and study their properties. The main motivation is the same: we replace the normalization $\overline{R}$ with module-finite algebras which are Maximal Cohen-Macaulay over the base ring and we study compactness properties in these topologies in the case of higher Krull dimension.

Finally, in the last section, we consider finite representation type and draw lattices of closed sets for the Kolmogorov quotients of these topological spaces - which have finitely many closed subsets.

\textbf{Acknowledgements.} We would like to thank the organizers of the Mathematics Research Program which took place in Istanbul Mathematical Sciences Center (IMBM) at the beginning of September 2021 and also a big thank you to the Centre International de Mathématiques Pures et Appliquées (CIMPA) for funding this program. This project is a fruit of this program the aim of which was to introduce undergraduate students from all over Turkiye to research level problems in mathematics. The fourth author is grateful for the opportunity to give a mini lecture series in this program where he met the first three authors.

\section{Background and Definitions}

Given a set $S$, a \textit{preorder} or a \textit{quasiorder} on $S$ is a binary relation $\leq$ on $S$ which is reflexive and transitive. For any point $a \in S$, the closure $\cl(a)$ is defined as $\cl(a) = \{ x \colon x \leq a \}$ and for a subset $U \subseteq S$, the closure of $U$ is the union
\begin{align*}
    \cl(U) = \bigcup_{a \in U} \cl(a).
\end{align*}
This operator is a Kuratowski closure operator and defines a topology on $S$ where a subset $U$ is closed if and only if $\cl(U) = U$. This topology is known as \textit{Alexandrov topology} or the \textit{specialization topology} associated to $\leq$. We note that arbitrary unions of closed subsets is again closed in this topology. Conversely, given a topological space where arbitrary unions of closed subsets stays closed, it comes from a preorder.

Since we only require reflexivity and transitivity from a preorder, we may have two distinct points $a,b \in S$ such that $a\leq b$ and $b \leq a$. In the Alexandrov topology, these points would be topologically indistinguishable: every closed subset that contains $a$ would also contain $b$ and vice versa. Hence an Alexandrov space does not have to be a $T_0$-space. Being topologically indistinguishable is an equivalence relation. The corresponding quotient space is sometimes called the \textit{Kolmogorov quotient} and it is a $T_0$-space. On the level of preordered sets, this identification gives us a partially ordered set. The Kolmogorov quotient of an Alexandrov space carries almost every topological information about the space. In particular, they are homotopy equivalent.

We are interested in a certain preorder defined on the set of isomorphism classes of a category (we assume that our categories have small skeletons). Given a commutative Noetherian ring $R$, we consider an $R$-linear category $\TT$. In particular, we assume that the Hom-sets have the structure of an $R$-module, compositions respect this structure and also for any two objects $M, N$, the Hom-set $\Hom_\TT(M,N)$ has the natural structure of an $\End_\TT(M)$-$\End_\TT(N)$-bimodule. Moreover, we assume that we have finite direct sums in $\TT$ and the Krull-Remak-Schmidt property is satisfied. Given a subcategory $\XX$, we denote its additive closure by $\add \XX$ - this is the smallest full subcategory of $\TT$ which is closed under finite direct sums and direct summands. For an object $M \in \TT$, we define the \textit{annihilator} of $M$ as
\begin{align*}
    \ann_R M := \ann_R \End_\TT(M).
\end{align*}
Then, the relation $N \leq M$ if and only if $\ann_R N \subseteq \ann_R M$ defines a preorder on the set of isomorphism classes of objects of $\TT$. We denote the corresponding Alexandrov space by $\alex(\TT)$ and the corresponding Kolmogorov quotient by $\kalex(\TT)$.

\begin{remark}
We note that by passing to the Kolmogorov quotient $\kalex(\TT)$, we are essentially identifying each point with an ideal of $R$. Hence, we can consider this topological space as a subset of all ideals of $R$. 
\end{remark}

We record some properties of the closure operator in this topology. Recall from the beginning of this subsection is that we have
\begin{align*}
    \cl(M) = \left\lbrace N \in \TT \colon \ann_R N \subseteq \ann_R M \right\rbrace.
\end{align*}

\begin{lemma}\label{basic-prop-lemma} Let $M, N \in \TT$.
\begin{enumerate}
    \item We have $\cl(M \oplus N) = \cl(M) \cap \cl(N)$.
    \item If $N \in \add M$, then $M \in \cl(N)$.
    \item For any auto-equivalence $\Sigma$ of $\TT$, we have $\cl(M) = \cl(\Sigma M)$. In particular, for any integer $n$, $\Sigma^n M$ is topologically indistinguishable from $M$.
    \item If $N \in \cl(M)$, then $\cl(N) \subseteq \cl(M)$.
\end{enumerate}
\end{lemma}
\begin{proof}
The first assertion follows from the fact that 
\begin{align*}
    \End_\TT(M\oplus N) \cong \End_\TT(M) \oplus \End_\TT(N) \oplus \Hom_\TT(M,N) \oplus \Hom_\TT(N,M)
\end{align*}
as $R$-modules. This implies that $\ann_R(M \oplus N) = \ann_R M \cap \ann_R N$. Hence, we have $\ann_R L \subseteq \ann_R(M\oplus N)$ if and only if $\ann_R L \subseteq \ann_R M$ and $\ann_R L \subseteq \ann_R N$. The second assertion follows from the first and the third assertion is due the fact that $N$ and $\Sigma(N)$ have isomorphic endomorphism rings. Finally, the last assertion follows from the transitivity of inclusion.
\end{proof}

This lemma has the following immediate consequence.

\begin{lemma}
Let $\XX$ be a subset of objects in $\TT$. Then, we have
\begin{align*}
    \cl(\XX) = \bigcup_{x \in \; \mathrm{ind}\XX} \cl(x)
\end{align*}
where $\mathrm{ind}\XX$ is the set of indecomposable objects in $\XX$.
\end{lemma}

\begin{lemma}\label{closed-sets-intersect}
The intersection of any two (consequently, finitely many) nonempty closed subsets is nonempty.
\end{lemma}

\begin{proof}
Let $C_1, C_2$ be two nonempty closed subsets with $M_1 \in C_1$ and $M_2 \in C_2$. This tells us that $\cl(M_1) \subseteq C_1$ and $\cl(M_2) \subseteq C_2$. Then, by Lemma \ref{basic-prop-lemma}, we have that $M_1 \oplus M_2$ is in the intersection $ \cl(M_1) \cap \cl(M_2) \subseteq C_1 \cap C_2$. The case of finitely many nonempty closed subsets follows similarly.
\end{proof}

\begin{proposition}\label{connected-prop}
The Alexandrov space $\alex(\TT)$ of an $R$-linear category $\TT$ is connected.
\end{proposition}
\begin{proof}
Recall that a topological space is connected if it can not be written as a disjoint union of two nonempty closed subsets. By Lemma \ref{closed-sets-intersect}, we know that there are no disjoint nonempty closed subsets. Result follows.
\end{proof}

We are now going to focus on our motivating question. We are interested in the existence of a minimum object in $\TT$ with the preorder we defined. To this end, let us denote by $\ann \TT$ the intersection
\begin{align*}
    \ann \TT = \bigcap_{M \in \TT} \ann_R M.
\end{align*}
Then, our question becomes: is there a single object $M$ such that $\ann \TT = \ann_R M$? In terms of the topology, we can restate this question as whether the intersection of \textit{all} closed subsets is nonempty. It is easy to see that the answer is affirmative if $\TT$ has only finitely many indecomposable objects up to isomorphism. Indeed, in this case, one can take the direct sum of a copy of each indecomposable object and our foregoing discussion tells us that this direct sum belongs to the intersection of all closed subsets. In the last section of this paper, we study the case of finite Cohen-Macaulay type. In particular, we look at the stable category of maximal Cohen-Macaulay modules over simple singularities. However, when there are infinitely many nonisomorphic indecomposable objects, we are not allowed to do this. Hence, we look at the next best thing: compactness.

Recall that a family of subsets of a topological space has the \textit{finite intersection property} if any finite subfamily has nonempty intersection. This is a useful definition to give the following characterization of compactness: a topological space is compact if and only if every family of closed subsets having the finite intersection property has nonempty intersection. But we know, by Lemma \ref{closed-sets-intersect}, that \textit{any} collection of finitely many closed subsets has nonempty intersection. In other words, any family of closed subsets has the finite intersection property. This is the main theorem of this section: our motivating question has an affirmative answer if and only if our space is compact.

\begin{theorem}\label{compact-theorem}
The Alexandrov space $\alex(\TT)$ of an $R$-linear category $\TT$ is compact if and only if there is a single object $M$ such that $\ann \TT = \ann_R M$.
\end{theorem}

Given a function $f: S_1 \to S_2$ between preordered sets, $f$ is a continuous function between the associated Alexandrov spaces if and only if it is monotonous. In our setting, it translates to the following lemma.

\begin{lemma}\label{main-continuity-lemma}
Let $\TT$ be an $R$-linear category and $\mathcal{S}$ be an $S$-linear category. Then, a function $F: \alex(\TT) \to \alex(\mathcal S)$ is continuous if and only for any $M , N \in \TT$ we have
\begin{align*}
    \ann_R M \subseteq \ann_R N \implies \ann_S F(M) \subseteq \ann_S F(N).
\end{align*}
\end{lemma}
For instance, if there is a fully faithful embedding from $\TT$ to $\mathcal{S}$, it induces a continuous embedding $\alex(\TT) \to \alex(\mathcal{S})$.

\section{The Alexandrov Space of Maximal Cohen-Macaulay Modules}
We are now going to restrict to our special setting of maximal Cohen-Macaulay modules over Gorenstein rings. The category we are interested in is the stable category of maximal Cohen-Macaulay modules. In this section, we will give some background on them and study continuous functions. We start by recalling the theory of maximal Cohen-Macaulay modules over Gorenstein rings. We refer to \cite{B} for details. From now on, we will assume that all modules are finitely generated.
		
Let $R$ be a commutative Noetherian local ring and $M$ be an $R$-module. A sequence $ x_1, \ldots, x_n $ of elements from the maximal ideal of $R$ is called a \textit{regular sequence} on $M$ if $(x_1, \ldots,x_n)M \neq M$ and $x_{i+1}$ is a nonzerodivisor on $M/(x_1, \ldots, x_i)M$ for any $i = 1, \ldots, n$. The \textit{depth} of a nonzero module $M$ is defined as the maximum length of regular sequences on $M$. This number is bounded above by the Krull dimension of $ R $ unless $M$ is the zero module whose depth is $\infty$. An $R$-module is called a \textit{maximal Cohen-Macaulay} module if its depth is greater than or equal to the Krull dimension of $R$.
		
From now on, we will assume that the ring $R$ is (strongly) Gorenstein which means that the regular module $R$ has finite injective dimension. In this case, this number equals the Krull dimension of the ring. We will denote the full subcategory of maximal Cohen-Macaulay $R$-modules by $\MCM(R)$.

For two $R$-modules $M,N$, denote by $P(M,N)$ the set of all $R$-module homomorphisms from $M$ to $N$ which factor through a projective module. It is a submodule of $\Hom_R(M,N)$ and we denote the quotient by
		\begin{align*}
		\sHom_R(M,N) = \Hom_R(M,N) / P(M,N).
		\end{align*}
With $\sMCM(R)$, we denote the category whose objects are the same as $\MCM(R)$ and whose morphism sets are $\sHom_R(M,N)$. We call it the \textit{stable category of maximal Cohen-Macaulay modules}. Note that this construction identifies projective modules with the zero object. On this category the syzygy functor $\Omega_R$ has an inverse, namely the cosyzygy functor $\Omega_R^{-1}$. $\sMCM(R)$ is a triangulated category with shift functor $ \Omega_R^{-1} $ and we have the following equivalence of triangulated categories
		\begin{align*}
	\sMCM(R) \cong D_{sg}(R) := D^b(\module R)/\mathrm{perf}(R).
		\end{align*}
where $ D^b(R) $ is the bounded derived category and $\mathrm{perf}(R)$ is the subcategory of perfect complexes. The quotient category $D_{sg}(R)$ is called the singularity category and it is nontrivial if and only if $R$ is not regular. So, it measures how singular the corresponding space is.

We denote by $\sann_R M$ the \textit{stable annihilator} of $M$ defined as the annihilator of the stable endomorphism ring $\sEnd_R(M)$. We call an element $r \in R$ a \textit{cohomology annihilator} if it belongs to $\sann_R M$ for every $M \in \MCM(R)$. The ideal consisting of all cohomology annihilators is called the \textit{cohomology annihilator ideal} and we denote it by $\ca(R)$.

In this section, we are interested in the Alexandrov space of maximal Cohen-Macaulay modules. That is, we will consider the topology we defined in the previous section on the stable category of maximal Cohen-Macaulay modules over a Gorenstein ring $R$. We denote this space by $\MM(R)$. Theorem \ref{compact-theorem} reads as follows in this case: $\MM(R)$ is compact if and only if there is a maximal Cohen-Macaulay $R$-module $M$ such that $\ca(R) = \sann_R(M)$. The following example is our motivating example.

\begin{example}\label{example-dimension-one}
If $R$ is a one dimensional complete reduced Gorenstein local ring, then $\MM(R)$ is a compact topological space. In this case, the normalization $\overline{R}$ has the property $\sann_R \overline{R} = \ca(R)$ \cite[Theorem 4.3]{Esentepe}.
\end{example}

\begin{example}\label{example-higher-dimension}
Let $k$ be a field and $X = (x_{ij})$ be the generic $n \times n$ square matrix. Consider the $n^2 - 1$ dimensional hypersurface ring $R = k[X]/\det(X)$. Let $M = \coker X$. Then, $M$ is a maximal Cohen-Macaulay $R$-module and $\sann_R M = \ca(R)$ \cite[Example 3.4]{Esentepe}. So, $\MM(R)$ is a compact topological space.
\end{example}

\subsection{Cohomology Annihilators and Continuous Maps}

In this subsection, we will see the relation between cohomology annihilators and continuity. In particular, our purpose is to prove the following theorem. We assume that $R$ is a Gorenstein local ring.

\begin{theorem}\label{continuity-theorem}
Let $x \in \ca(R)$ be a nonzerodivisor. The function $\pi: \MM(R) \to \MM(R/xR)$ defined by $M \mapsto M/xM$ is continuous.
\end{theorem}

The following lemma is the version of Lemma \ref{main-continuity-lemma} for our case.

\begin{lemma}\label{continuity-lemma-1}
Let $R$ and $S$ be two Gorenstein local rings and $\Phi: \MM(R) \to \MM(S)$ any function. Then, $\Phi$ is continuous if and only if for any two maximal Cohen-Macaulay $R$-modules $L, M$ we have
\begin{align*}
    \sann_R L \subseteq \sann_R M \implies \sann_S \Phi(L) \subseteq \sann_S \Phi(M).
\end{align*}
\end{lemma}

We will now consider such maps as in Lemma \ref{continuity-lemma-1} which are induced by ring homomorphisms. Let $R$ and $S$ be as above and consider a ring map $f: R \to S$. Then, $f$ induces a map $\Phi_f : \MM(R) \to \MM(S)$ which takes a maximal Cohen-Macaulay $R$-module $M$ to the MCM-approximation (or stabilization) of $M \otimes_R S$ in $\sMCM(S)$.

Suppose that we have a ring element $r$ which stably annihilates $M$. This means that the multiplication map by $r$ factors through a free $R$-module and we have a commutative diagram

\begin{align*}
 	\xymatrix{
 		M \ar[rr]^{r} \ar[dr]_{\alpha} && M  \\
 		& P \ar[ur]_{\beta} 
 	}
 \end{align*}
where $P$ is a free $R$-module. Tensoring with $S$, we have the commutative diagram
\begin{align*}
 	\xymatrix{
 		M \otimes_R S \ar[rr]^{f(r)} \ar[dr]_{\alpha \otimes_R S} && M\otimes_R S  \\
 		& Q \ar[ur]_{\beta \otimes_R S} 
 	}
 \end{align*}
 with $Q$ a free $S$-module. We can see this as a commutative diagram in the module category of $S$ or we can see it in the derived category of $S$. Note that $M \otimes_R S$ is not necessarily maximal Cohen-Macaulay as an $S$-module but we know that multiplication by $f(r)$ factors through a perfect complex. Hence, when we pass to the Verdier quotient $D_{sg}(R)$, multiplication by $f(r)$ induces the zero map. Using Buchweitz's equivalence, we can conclude that $f(r) \in \sann_S \Phi_f(M)$. We record this as our second lemma.
 
 \begin{lemma}\label{contunuity-lemma-2}
 Let $f: R \to S$ be a ring homomorphism of Gorenstein local rings and consider the map $\Phi_f: \MM(R) \to \MM(S)$ which takes $M$ to the MCM-approximation of $M \otimes_R S$. Then, we have
 \begin{align*}
     f(\sann_R M) \subseteq \sann_S \Phi_f(M)
 \end{align*}
 for any maximal Cohen-Macaulay module $M$.
 \end{lemma}

Now, we are ready to combine the two lemmas: Assume that the inclusion in Lemma \ref{contunuity-lemma-2} is actually an equality. Then, we would have an implication \begin{align*}
\sann_R L \subseteq \sann_R M \implies f(\sann_R L) \subseteq f(\sann_R M)
\end{align*}
which by our assumption would give
\begin{align*}
\sann_R L \subseteq \sann_R M \implies \sann_S \Phi_f(L) \subseteq \sann_S \Phi_f(M)
\end{align*}
concluding that $\Phi_f$ is continuous by Lemma \ref{continuity-lemma-1}. Let us record this.
\begin{proposition}\label{continuity-prop-3}
Let $f: R \to S$ and $\Phi_f$ be as above. If we have an equality
\begin{align*}
         f(\sann_R M) = \sann_S \Phi_f(M)
\end{align*}
for every maximal Cohen-Macaulay $R$-module, then $\Phi_f$ is continuous.
\end{proposition}
This way, we have reduced our question to investigate the equality in the preceeding proposition. We will now give an example where this equality holds. 

The following lemma follows from \cite[Lemma 2.1, Lemma 2.2]{dugasleuschke} and we will need it for our example.
\begin{lemma}
If $R$ is a Gorenstein local ring and $x$ is a cohomology annihilator which is a nonzerodivisor, then for any maximal Cohen-Macaulay $R$-module $M$, we have
\begin{align*}
    \Omega_R(M/xM) \cong M \oplus \Omega_R(M).
\end{align*}
\end{lemma}
Now, we consider a Gorenstein local ring $R$, a nonzerodivisor $x \in \ca(R)$ and the quotient map $\pi: R \to R/xR$. By abuse of notation, for an $R$-module $M$, we denote by $\pi M$ the $R/xR$-module $M/x M$. Then, for any maximal Cohen-Macaulay $R$-module $M$, we know that $\pi M$ is a maximal Cohen-Macaulay $R/xR$-module. Therefore, $\Phi_\pi(M) = \pi M = M \otimes_R (R/xR)$. Hence, by Proposition \ref{continuity-prop-3}, if we can show that $ \sann_{R/xR} \pi M \subseteq \pi(\sann_R M)$, then $\pi$ will be continuous.

\begin{lemma}
With the above notation, we have $\sann_{R/xR} \pi M \subseteq \pi(\sann_R M)$.
\end{lemma}
\begin{proof}
Let $M$ a maximal Cohen-Macaulay $R$-module. Firstly, if $q \in \sann_{R/xR}(\pi M)$, then $q \in \ann_{R/xR}\sHom_{R/xR}(Z,\pi M) $ for any $Z \in \MCM(R/xR)$. In particular, we can take $Z = \Omega_{R/xR}\pi \Omega_R^{-1}M$. Then, $Z$ is still maximal Cohen-Macaulay as an $R/xR$-module. Now,
\begin{align*}
    \sHom_{R/xR}(Z,\pi M) & \cong \Ext_{R/xR}^1(\pi \Omega_R^{-1} M, \pi M) \cong \Ext_R^2(\pi \Omega_R^{-1}M, M)  \\ & \cong \Ext_R^1(\Omega_R \pi \Omega_R^{-1}M, M) \cong \Ext_R^1( \Omega_R^{-1}M \oplus M, M) \\
    & \cong \sHom_R(\Omega_R^{-1}M,M) \oplus \sEnd_R(M).
\end{align*}
We know that $q$ annihilates $\sEnd_{R/xR}(\pi M)$ as an $R$-module. As a result, we get that 
\begin{align*}
    q \in \ann_R [ \sHom_R(\Omega_R^{-1}M,M) \oplus \sEnd_R(M) ] = \ann_R \sEnd_R(M) = \sann_R M
\end{align*}
as required.
\end{proof}
Thus, we have proved Theorem \ref{continuity-theorem}.

We finish this subsection by noting other continuous functions which are not very interesting. For example, the functions $\MM(R) \to \MM(R)$ defined by $M \mapsto \Omega_R M$ and $M \mapsto \Hom_R(M,R)$ are continuous but they are not very interesting as maps between topological spaces because they fix closed sets, hence trivially continuous. Indeed, as previously discussed, we have equalities $\sann_R M = \sann_R \Omega_RM = \sann_R\Hom_R(M,R)$. On the other hand, for a fixed maximal Cohen-Macaulay module $X$, the map $M \mapsto M \oplus X$ is continuous as $\sann_R L \subseteq \sann_R M$ implies $\sann_R L \cap \sann_R X \subseteq \sann_R M \cap \sann_R X$. And if $f$ and $g$ are continuous functions, the map $f \oplus g$ taking $M$ to $f(M) \oplus g(M)$ is also continuous. This follows once again from the relationship between stable annihilator ideals and direct sums.

\section{Orders and Cohen-Macaulay Modules}

Recall that our motivating question was whether there exists a single maximal Cohen-Macaulay module $M$ such that $\ca(R) = \sann_R(M)$ which was the case for a one dimensional Gorenstein ring $R$ (see Example \ref{example-dimension-one}). Indeed, we saw that in this case we can take $M = \overline{R}$, the normalization of $R$. This is not only a single maximal Cohen-Macaulay $R$-module but a module finite $R$-algebra (of global dimension 1) which turns out to be maximal Cohen-Macaulay as an $R$-module. In this section, we will investigate whether this can be generalized.

Recall that a module-finite algebra over a Cohen-Macaulay local ring $R$ is called an $R$-\textit{order} if it is maximal Cohen-Macaulay as an $R$-module. We will always assume that our orders have finite injective dimension as modules on both sides. Hence, they will be \textit{Iwanaga-Gorenstein} rings.

For the rest of this section, we will assume that $\Lambda$ is an $R$-order where $R$ is a Gorenstein local ring. We consider two subcategories of $\module \Lambda$ - the category of finitely generated $\Lambda$-modules. Firstly, we consider $\MCM(\Lambda)$: the category of maximal Cohen-Macaulay $\Lambda$-modules. A $\Lambda$-module $X$ is called \textit{maximal Cohen-Macaulay} if $\Ext_\Lambda^i(X,\Lambda) = 0$ for all $i > 0$. As in the commutative case, this is a Frobenius category and its stable category is naturally triangulated and equivalent to the singularity category of $\Lambda$. Next, we consider $\ZCM(\Lambda)$: the category of \textit{centrally Cohen-Macaulay} $\Lambda$-modules. A $\Lambda$-module $X$ is called \textit{centrally Cohen-Macaulay} if it is maximal Cohen-Macaulay as an $R$-module. We note that $\ZCM(\Lambda)$ is independent of the Cohen-Macaulay local ring $R$ over which $\Lambda$ is a module-finite algebra. We say that $\Lambda$ is an $n$-\textit{canonical order} if its \textit{canonical module} $\omega = \Hom_R(\Lambda, \omega_R)$ has projective dimension $n$ as a $\Lambda$-module. Note that $\omega_R \cong R$ in our case as we assumed $R$ is Gorenstein. We write $D(-)$ for the functor $\Hom_R(-,R)$ so that $\omega= D\Lambda$.

Every maximal Cohen-Macaulay $\Lambda$-module is also centrally Cohen-Macaulay and when $\Lambda$ is an $n$-canonical order, there is an equality $\MCM(\Lambda) = \Omega_\Lambda^n \ZCM(\Lambda)$ where $\Omega_\Lambda$ is the syzygy operator on $\module \Lambda$. When $\Lambda$ is $0$-canonical, we call it a \textit{Gorenstein order}. In this case, maximal Cohen-Macaulay and centrally Cohen-Macaulay modules coincide. When $\Lambda$ is an $n$-canonical order of finite global dimension, then $\gldim \Lambda = \dim R + n$ and in this case every maximal Cohen-Macaulay $\Lambda$-module is projective. When $\Lambda$ is a Gorenstein order of finite global dimension, then every centrally Cohen-Macaulay $\Lambda$-module is projective, as well, and in this case $\Lambda$ is called a \textit{nonsingular order}. In this sense, maximal Cohen-Macaulay modules measure the \textit{singularity} of $\Lambda$ by finiteness of global dimension whereas centrally Cohen-Macaulay modules measure whether $\Lambda$ has the lowest possible global dimension.

\subsection{Dimension Two}

Let us restrict ourselves to the case of Krull dimension two. That is, for this subsection, let us assume that $R$ is a two dimensional Gorenstein local ring. In this case, assume that we are given a short exact sequence
\begin{align*}
    0 \to X \to Y \to Z \to 0
\end{align*}
of $R$-modules with $Y \in \MCM(R)$. Then, applying $D$ to this exact sequence, we get a long exact sequence
\begin{align*}
    0 \to DZ \to DY \to DX \to \Ext_R^1(Z,R) \to 0
\end{align*}
since $\Ext_R^1(Y,R) = 0$ and thus, by depth lemma, we conclude $DZ$ has depth at least $2$ proving that $DZ$ is maximal Cohen-Macaulay for any $R$-module $Z$.

\begin{proposition}\label{order-proposition}
Let $R$ be a two dimensional Gorenstein local ring and $\Lambda$ be an $R$-order such that the structure map $R \to \Lambda$ is a split monomorphism. Then, for any closed subset $\XX \subseteq \MM(R)$, we have
\begin{align*}
    \ZCM \Lambda \cap \XX \neq \emptyset.
\end{align*}
\end{proposition}
\begin{proof}
Let $M$ be a maximal Cohen-Macaulay $R$-module and consider the $\Lambda$-module $M \otimes_R \Lambda$. Since the structure map is a split monomorphism, $R$ is a direct summand in $\Lambda$ and thus, $M$ is a direct summand in $M \otimes_R \Lambda$ as an $R$-module. Note that $M \otimes_R \Lambda$ is not necessarily maximal Cohen-Macaulay as an $R$-module. So, we consider the reflexive hulls $D^2M \cong M$ and $D^2(M \otimes_R \Lambda)$ which we know to be maximal Cohen-Macaulay by our discussion above. Since $D^2(M \otimes_R \Lambda)$ is a $\Lambda$-module which is maximal Cohen-Macaulay as an $R$-module it belongs to $\ZCM \Lambda$. Moreover, since $M$ is a direct summand of $D^2(M \otimes_R \Lambda)$, we see that by Lemma \ref{basic-prop-lemma}
\begin{align*}
    D^2(M \otimes_R \Lambda) \in \cl(M).
\end{align*}
This shows that $\ZCM \Lambda \cap \cl(M) \neq \emptyset$ for any $M \in \MCM(R)$. Since closed sets are unions of sets of the form $\cl(M)$, we are done.
\end{proof}

Now, we have a two dimensional analogue of Example \ref{example-dimension-one}.

\begin{corollary}\label{order-corollary}
Let $R$ be a two dimensional Gorenstein local ring and $\Lambda$ be a nonsingular $R$-order such that the structure map $R \to \Lambda$ is a split monomorphism. Then, for any closed subset $\XX \subseteq \MM(R)$, we have $\Lambda \in \XX$.
\end{corollary}

In terms of our topological spaces, we have the following corollary.

\begin{theorem}\label{example1}
If $R$ is a two dimensional Gorenstein local ring admitting a nonsingular order for which the structure map is a split monomorphism, then $\MM(R)$ is compact.
\end{theorem}

\subsection{Variations of Our Topology}

In order to prove analogues of Example \ref{example-dimension-one}, Proposition \ref{order-proposition} or Corollary \ref{order-corollary} in higher Krull dimensions, we will change the definition of our topology a little bit. Let $R$ be a Gorenstein local ring and $M$ be a maximal Cohen-Macaulay $R$-module. We start with defining a new operator
\begin{align*}
    \cl_n(M) = \left\lbrace L \in \sMCM(R) \colon (\sann_R L)^n \subseteq \sann_RM \right\rbrace
\end{align*}
for any positive integer $n$. Note that we have $\cl_n(M) \subseteq \cl_{n+1}(M)$ for any $n$, so we are making our closed sets bigger by considering the operator $\cl_n$ for $n \geq 2$. Earlier, we were interested in the case whether the intersection of all closed sets is nonempty. By making our closed sets bigger, we increase the likelihood of this. Of course, this comes with the price of weakening our motivating question.

We do not define the closure of a set as the union of $\cl_n(M)$ where the union is taken over objects $M$ in the set. The main reason for this is that $\cl_n$ does not define a Kuratowski closure operator in that case as $\cl_n \circ \cl_n \neq \cl_n$. Put in other words, the relation given by $M \leq N$ if and only if $(\sann_R M)^n \subseteq \sann_R N$ is not transitive.

\begin{definition}
We define the topological space $\MM(n,R)$ as the space whose closed sets are generated by closed sets of the form $\cl_n(M)$ with $M \in \MCM(R)$.
\end{definition}

We start by noting that although the first three assertions in Lemma \ref{basic-prop-lemma} hold for $\cl_n$, the last assertion may fail. That is, if $N \in \cl_n(M)$, this does not necessarily mean that $\cl_n(N) \subseteq \cl_n(M)$. We can only deduce that $\cl_{n^2}(N) \subseteq \cl_n(M)$. We will give examples of this in the next section. However, we note that this has the consequence that $\cl_n(N)$ is not necessarily the (topological) closure of $N$ in $\MM(n, R)$. Indeed, suppose that we have $N \in \cl_n(M)$. Then, we have $N \in \cl_n(N) \cap \cl_n(M)$. This intersection is a closed set that contains $N$ but it \textit{may} be strictly smaller than $\cl_n(N)$. Because we may have $\cl_n(N) \nsubseteq \cl_n(M)$.

Our proof for Lemma \ref{closed-sets-intersect} does not work for the reason stated in the previous paragraph. However, we have the following version of it: Let $n$ be a positive integer. For any two maximal Cohen-Macaulay $R$-modules $M$ and $N$, we have $\cl_n(N) \cap \cl_n(M) \neq 0$. Indeed, $M \oplus N$ belongs to the intersection. Consequently, the collection of all closed subsets of the form $\cl_n(M)$ has the finite intersection property. Moreover, we note that if there exists a single object $X$ such that $X \in \cl_n(M)$ for all $M$, then $X$ belongs to every closed set. Combining all this information, we have the following version of Theorem \ref{compact-theorem}.

\begin{theorem}
The topological space $\MM(n,R)$ is compact if and only if there exists a maximal Cohen-Macaulay $R$-module $X$ with $(\sann_R X)^n \subseteq \ca(R)$.
\end{theorem}

Let $M$ be a maximal Cohen-Macaulay module. By the \textit{generation time} of $M$, we mean the number of cones required to generate $\sMCM(R)$, up to syzygies and direct summands and we denote it by $g(M)$. If there is an $M$ with finite generation time, then by definition $g(M)$ is an upper bound for the \textit{dimension} of $\sMCM(R)$, a notion introduced by Raphael Rouquier \cite{Rouquier}. It is proved in \cite[Section 2.1]{Esentepe} that if $M$ is a maximal Cohen-Macaulay $R$-module with finite generation time, then $(\sann_R M)^{g(M)} \subseteq \ca(R)$. The proof depends on the following lemma which we will record as it is needed later.

\begin{lemma}\label{sann-lemma}
  Let $M_1,M_2,M_3$ be maximal Cohen-Macaulay $R$-modules. If there is a short exact sequence $0 \to M_1 \to M_2 \to M_3 \to 0$ of $R$-modules, then
      \begin{align*}
      \sann_R M_i \cdot \sann_R M_j \subseteq \sann_R M_k 
      \end{align*}
      where $i,j,k$ are distinct elements of the set $\{1,2,3\}$.
\end{lemma}

As a corollary, we have the following.

\begin{corollary}\label{Rouquier-corollary}
 The topological space $\MM(n,R)$ is compact for any $n \geq \dim \sMCM(R)$.
\end{corollary}

\begin{example}
Let us assume that $R=(R, \m, k)$ is a $d$-dimensional equicharacteristic excellent Gorenstein local ring with an isolated singularity. In this case, the cohomology annihilator ideal is $\m$-primary and we know that $\Omega^d(k)$ is maximal Cohen-Macaulay with generation time at most $(\nu - d +1 ) + \ell\ell(R/\ca(R))$ where $\nu$ is the minimal number of generators for $\ca(R)$ and $\ell\ell$ denotes the \textit{Loewy length} \cite{Dao-Takahashi}. Hence, in this case, $\MM(n,R)$ is compact for any $n \geq (\nu - d +1 ) + \ell\ell(R/\ca(R))$.
\end{example}

\begin{proposition}
Let $\Lambda$ be an $R$-order and $M \in \ZCM \Lambda$. Then, for every positive integer $i$, we have an inclusion 
\begin{align*}
    (\sann_R \Lambda )^i \; . \;  \sann_R \Omega^i_\Lambda M \subseteq \sann_R M.
\end{align*}
\end{proposition}
\begin{proof}
Proof is by induction on $i$. Let $0 \to \Omega_\Lambda M \to F \to M \to 0$ be a short exact sequence defining $\Omega_\Lambda M$ with $F$ a projective $\Lambda$-module. Note that by the Depth Lemma, we have $\Omega_\Lambda M$ is also a centrally Cohen-Macaulay module. Then by Lemma \ref{sann-lemma}, we see that
\begin{align*}
    \sann_R \Lambda \; . \; \sann_R\Omega_\Lambda M \subseteq \sann_RM
\end{align*}
which proves the case for $i = 0$. Now, the proof is complete by observing the same for the short exact sequence
\begin{align*}
    0 \to \Omega^{i+1}_\Lambda M \to G \to \Omega^i_\Lambda M \to 0
\end{align*}
with $G$ a projective $\Lambda$-module.
\end{proof}

\begin{corollary}
Let $\Lambda$ be an $n$-canonical $R$-order of finite global dimension. Then, for any centrally Cohen-Macaulay $\Lambda$-module $X$, we have 
\begin{align*}
    \MCM(\Lambda) \cap \cl_{n+1}(X) \neq \emptyset.
\end{align*}
In particular, if $\Lambda$ has finite global dimension, then we have $\Lambda \in \cl_n(X)$.
\end{corollary}
\begin{proof}
The first part of the proof follows from the preceding lemma. Indeed, if $X$ is a centrally Cohen-Macaulay $\Lambda$-module, we have that $\Omega_\Lambda^n X$ is a maximal Cohen-Macaulay $\Lambda$-module. By considering $\Lambda \oplus \Omega^n_\Lambda X$, we get
\begin{align*}
    \sann_R(\Lambda \oplus \Omega^n_\Lambda X)^{n+1} = (\sann_R \Lambda \cap \sann_R \Omega^n_\Lambda X)^{n+1} \subseteq \sann_R \Lambda \; .\; \sann_R\Omega_\Lambda^n X \subseteq \sann_R X.
\end{align*}
If $\Lambda$ has finite global dimension, then every maximal Cohen-Macaulay $\Lambda$-module is projective. Thus, the second part follows.
\end{proof}

\begin{proposition}
Let $\Lambda$ be an $R$-order and assume that the structure map $R \to \Lambda$ is a split monomorphism. Then, for any $M \in \MCM(R)$, we have
\begin{align*}
    \ZCM(\Lambda) \cap \cl_{d+1}(M) \neq \emptyset.
\end{align*}
\end{proposition}
\begin{proof}
Let $M$ be a maximal Cohen-Macaulay $R$-module and consider the $\Lambda$-module $M \otimes_R \Lambda$. Now, we proceed by induction on codepth of $M \otimes_R \Lambda$. If the codepth of $M \otimes_R \Lambda$ is zero, then it is centrally Cohen-Macaulay and we are done. If not, we consider the short exact sequence
\begin{align*}
    0 \to \Omega_\Lambda(M \otimes_R \Lambda) \to F \to M \otimes_R \Lambda \to 0
\end{align*}
with $F$ a free module. While $M \otimes_R \Lambda$ is not necessarily a centrally Cohen-Macaulay module, we can consider the maximal Cohen-Macaulay approximation of it inside $\sMCM(R)$ and we have a triangle
\begin{align*}
    [\Omega_\Lambda(M \otimes_R \Lambda)]^{st} \to F \to (M \otimes_R \Lambda)^{st}
\end{align*}
which tells us that 
\begin{align*}
    \sann_R [\Omega_\Lambda(M \otimes_R \Lambda)]^{st} . \sann_R \Lambda \subseteq  \sann_R (M \otimes_R \Lambda)^{st}.
\end{align*}
Note that by the Depth Lemma, taking the syzygy as a $\Lambda$-module reduces codepth by $1$. Thus, we are done by induction after taking $d$ many syzygies. In particular, the $\Lambda$-module $\Lambda \oplus \Omega^d_\Lambda(M\otimes_R \Lambda)$ is a module in the intersection which we prove to be nonempty.
\end{proof}

By similar arguments, we also have the following theorem. We note that if $\Lambda$ is an $n$-canonical $R$-order, then the injective dimension of $\Lambda$ is $d + n + 1$ as a module over itself (right or left). 

\begin{theorem}
Let $\Lambda$ be an $n$-canonical $R$-order and assume that the structure map $R \to \Lambda$ is a split monomorphism. Then, for any $M \in \MCM(R)$, we have
\begin{align*}
    \MCM(\Lambda) \cap \cl_{d+n+1}(M) \neq \emptyset.
\end{align*}
In particular, if $\Lambda$ is a Gorenstein order, we have $\MCM(\Lambda) \cap \cl_{d+1}(M) \neq \emptyset$ and if it is a nonsingular order we have $\Lambda \in \cl_{d+1}(M)$.
\end{theorem}
We end this section by stating these results in the language of our topology.
\begin{corollary}\label{example3}
If $R$ admits an $n$-canonical $R$-order of finite global dimension, then for any $m \geq \gldim \Lambda + 1$, the topological space $\MM(m,R)$ is compact.
\end{corollary}

\section{Finite Representation Type}

\begin{chunk}
\textbf{Double branched covers.} Let $k$ be an algebraically closed field of characteristic zero, $(S,\mathfrak{n})$ be a ring of power series in $d+1$ variables, $f \in \mathfrak{n}$ is a power series in the maximal ideal and $R = S/(f)$ be the hypersurface ring defined by $f$. Then, the double branched cover of  $R$ is the ring $R^\sharp = S^\sharp / (f+z^2)$ where $S^\sharp = S[z]$.

\begin{lemma}
The function $\pi: \MM(R^\sharp) \to \MM(R)$ defined by $M \mapsto M/zM$ is continuous.
\end{lemma}
\begin{proof}
Given a hypersurface ring defined by a power series $g$, the Jacobian ideal is contained in the cohomology annihilator ideal. In particular, all partial derivatives of $g$ are cohomology annihilators. This means that in our setting, $z \in \ca(R^\sharp)$ since the partial derivative of $f+z^2$ with respect to $z$ is $2z$. Now, the proof follows from Theorem \ref{continuity-theorem}.
\end{proof}

This lemma followed from Theorem \ref{continuity-theorem} which relied on the fact that for any maximal Cohen-Macaulay $R^\sharp$-module $M$, we have
\begin{align}\label{first-iso}
\Omega_{R^\sharp} (\pi M) \cong M \oplus \Omega_{R^\sharp} M.
\end{align}
As we have seen, this is a general phenomenon related to cohomology annihilators. Start with a module, annihilate the cohomology annihilator and then take the syzygy of the quotient module: you will not get the original module back but the direct sum of the original module and its syzygy. In the case of a double branched cover, a similar phenomenon occurs when one starts with a module over the quotient ring, as well. More precisely, keeping the notation above, for any $R$-module $N$, we have
\begin{align}\label{second-iso}
    \pi \Omega_{R^\sharp} N \cong N \oplus \Omega_R N
\end{align}
which was proved by Horst Kn{\"o}rrer in \cite{Knorrer}.
\begin{lemma}
The function $\Omega_{R^\sharp}: \MM(R) \to \MM(R^\sharp)$ defined by $N \mapsto \Omega_{R^\sharp}N$ is continuous.
\end{lemma}
\begin{proof}
Assume that $\sann_R M \subseteq \sann_R N$. We should show that $\sann_{R^\sharp} \Omega_{R^\sharp} M \subseteq \sann_{R^\sharp} \Omega_{R^\sharp} N$. Let $r \in \sann_{R^\sharp} \Omega_{R^\sharp} M$. We know that since $z$ is a cohomology annihilator of $R^\sharp$ we have $\pi(\sann_{R^\sharp} X) = \sann_R \pi X$ for every maximal Cohen-Macaulay $R^\sharp$-module $X$. So, we have $$\pi(r) \in \sann_R(\pi \Omega_{R^\sharp} M) = \sann_R(M \oplus \Omega_RM) = \sann_R M \subseteq \sann_R N.$$ From this, we conclude that $$\pi(r) \in \sann_R N = \sann_R (N \oplus \Omega_R N) = \sann_R(\pi \Omega_R^\sharp N). $$ Thus, $r \in \sann_{R^\sharp} \Omega_{R^\sharp} N$ as required.
\end{proof}
Now we established that the functions $\pi: \MM(R^\sharp) \to \MM(R)$ and $\Omega_{R^\sharp}: \MM(R) \to \MM(R^\sharp)$ are continuous. Moreover, the isomorphisms (\ref{first-iso}) and (\ref{second-iso}) tell us that $\pi$ and $\Omega_{R^\sharp}$ are \textit{almost} inverses of each other. If we instead consider these functions as functions between the corresponding Kolmogorov quotients, they actually become inverses of each other as $\sann M = \sann (M \oplus \Omega M)$. Hence, we conclude the following proposition.

\begin{proposition}\label{double-branched-homeomorphism}
Let $R$ be a hypersurface ring and $R^\sharp$ be its double branched cover. Then, the Kolmogorov quotients of $\MM(R)$ and $\MM(R^\sharp)$ are homeomorphic.
\end{proposition}
\end{chunk}

\begin{chunk}
\textbf{Simple singularities.} Now, let us focus on some Gorenstein local rings for which the Kolmogorov quotient of the Alexandrov space of maximal Cohen-Macaulay modules is a finite topological space. We are considering \textit{CM-finite} Gorenstein local rings which, by definition, have finitely many indecomposable maximal Cohen-Macaulay modules up to isomorphism. The classification of such rings were given in the 1980s by the work of Buchweitz-Greuel-Schreyer \cite{BGS} and Kn{\"o}rrer: if $R$ is a complete Gorenstein local ring containing the residue field $k$, which is an algebraically closed field of characteristic zero, then $R$ is CM-finite if and only if it is a \textit{simple singularity} meaning that it is of the form $R = k[[x,y,z_2,\ldots, z_d]]/(f) $ where $f$ is one of the following polynomials:
\begin{align*}
    (A_n): &\quad x^2 + y^{n+1} + z_2^2 + \ldots + z_d^2 \quad n \geq 1, \\
    (D_n): &\quad x^2y + y^{n-1} + z_2^2 + \ldots + z_d^2 \quad n \geq 4, \\
    (E_6): &\quad x^3+y^4 + z_2^2 + \ldots + z_d^2\\
    (E_7): &\quad x^3+xy^3+z_2^2 + \ldots + z_d^2\\
    (E_8): &\quad x^3+y^5+z_2^2 + \ldots + z_d^2
\end{align*}
We note that by Proposition \ref{double-branched-homeomorphism} it is enough to study one dimensional simple singularities for our purposes. We also note that by the work of Eisenbud \cite{Eisenbud}, every maximal Cohen-Macaulay module over a hypersurface ring has a 2-periodic projective (or complete) resolution. Therefore, if $R$ is a hypersurface ring and $M$ is a maximal Cohen-Macaulay $R$-module, then we have
\begin{align*}
    \sann_R M = \ann_R \sEnd_R(M) = \ann_R \Ext_R^2(M,M).
\end{align*}
\end{chunk}

\begin{chunk}
\textbf{Type A.} For an $A_n$ singularity, we can further use Proposition \ref{double-branched-homeomorphism} to reduce to the zero dimensional case where $R = k[y]/(y^{n+1})$. A complete list of indecomposable maximal Cohen-Macaulay modules up to isomorphism in this case is $M_0, \ldots, M_n$ where $M_i = k[y]/(y^{i+1})$.

A projective resolution of $M_i$ is
\begin{align*}
        \ldots R \xrightarrow{y^{i+1}} R \xrightarrow{y^{n-i}} R \xrightarrow{y^{i+1}} R \xrightarrow{y^{n-i}} R \xrightarrow{y^{i+1}} R \to  0.
    \end{align*}
    Then, we apply $\Hom_R(-,M_i)$ to this and we get
    \begin{align*}
        0 \to \Hom_R(R,M_i) \xrightarrow{\Hom_R(y^{i+1}, M_i)} \Hom_R(R,M_i) \xrightarrow{\Hom_R(y^{n-i}, M_i)} \to \ldots  
    \end{align*}
    Note that $\Hom_R(R,M_i) \cong M_i$ and the maps are just multiplication by $y^{i+1}$ and $y^{n-i}$ alternatively. Thus, we need to compute the homology of 
    \begin{align*}
        M_i \xrightarrow{y^{i+1}} M_i \xrightarrow{y^{n-i}} M_i.
    \end{align*}
    The image of the first map is zero. So, the homology is just the kernel of the map $M_i \xrightarrow{y^{n-i}} M_i$. So, we have
    \begin{align*}
        \Ext_R^2(M_i,M_i) = \ker( M_i \xrightarrow{y^{n-i}} M_i) = (y^{1+i-n+i}).
    \end{align*}
    From here, we get that $\sann_R(M_i) = (y^{i+1}, y^{n-i}) = (y^{\min{(i+1, n-i)}})$. 
    
    Hence, the topology of the Kolmogorov space of $\MM(R)$ is generated by the closures of the points $M_0, M_1, \ldots, M_t$ (where $t= \lceil n /2\rceil$) whose stable annihilators are $(y), (y^2), \ldots, (y^{n-t})$, respectively. In fact, these are all the closed sets in this space. Hence the topology on $A_n$ is given by the indecomposable $\sann(M_i)$'s which are totally ordered by inclusion as
   \begin{equation*}
       \sann(M_t)\subset \sann(M_{t-1})\subset \ldots \subset \sann(M_1)\subset \sann(M_0).
   \end{equation*}
\end{chunk}

\begin{chunk}
\textbf{Matrix factorizations.} 
Let $S$ be a regular local ring with the maximal ideal $\mathfrak{n}$ and let $f$ be a nonzero element of $\mathfrak{n}^2$. A \emph{matrix factorization} of $f$ is a pair $(\varphi, \psi)$ of homomorphisms $\varphi \colon F \to G,$ $\psi \colon G \to F$ of free $S$-module of the same rank such that $$\varphi\psi = f 1_F \text{ and } \psi\varphi = f 1_G.$$ Equivalently (after choosing bases), $\varphi$ and $\psi$ are square matrices of the same size over $S$, say $n \times n$, such that $$\varphi\psi = \psi\varphi = f I_n.$$

A homomorphism of matrix factorizations is a pair $(\alpha, \beta)$ of $S$-module homomorphisms between $(\varphi, \psi)$ and $(\varphi', \psi')$ rendering the diagram
\begin{align*}
    \xymatrix{
    F \ar[r]^\phi \ar[d]_\beta & G \ar[r]^\psi \ar[d]^\alpha & F \ar[d]^\beta \\
    F' \ar[r]^{\phi '} & G' \ar[r]^{\psi '} & F'
    }
\end{align*}

commutative. We denote the category of matrix factorizations of $f$ over $S$ by $\mathrm{MF}(S,f)$. The functor which takes a matrix factorization $(\phi, \psi)$ to $\coker (\phi)$ induces an equivalence of triangulated categories between the \emph{homotopy} category of matrix factorizations of $f$ over $S$ and the stable category of maximal Cohen-Macaulay modules over $R = S/(f)$ due to a theorem of Eisenbud \cite{Eisenbud} combined with a theorem of Buchweitz \cite{Buchweitz}.

Recall that an element $r \in R$ belongs to the stable annihilator of a maximal Cohen-Macaulay module $M$ if and only if \textit{multiplication by $r$} factors through a free $R$-module. This, by the equivalence above, is equivalent to saying that \textit{multiplication by $r$} on $(\phi, \psi)$ is \textit{nullhomotopic} which, by definition, means that there are maps $t,p$ such that $r I_n = \phi p + t \psi = p \phi + \psi t$. In what follows, we will use this to compute the stable annihilators of maximal Cohen-Macaulay modules over a $D_n$ singularity.
\end{chunk}

\begin{chunk}
\textbf{Type D.} For a $D_n$ type singularity with odd $n$, we have the indecomposable matrix factorizations in the following form \cite{Y}:
\begin{gather*}
    (\alpha, \beta) = (y, x^2+y^{n-2})\\
    \varphi_j= \begin{pmatrix}
    x & y^j \\
    y^{n-j-2} & -xy
    \end{pmatrix} \;,\;
    \psi_j= \begin{pmatrix}
    xy & y^{i+1} \\
    y^{n-j-1} & -xy
    \end{pmatrix} \;\\
    \xi_j= \begin{pmatrix}
    x & y^{i} \\
    y^{n-j-1} & -xy
    \end{pmatrix} \;,\;
    \eta_j= \begin{pmatrix}
    x & y^{i} \\
    y^{n-j-1} & -xy
    \end{pmatrix}
\end{gather*}

for $0 \leq j \leq n-3$. We denote $A = \coker(\alpha, \beta)$,  $M_j = \coker(\varphi_j, \psi_j)$, $X_j= \coker(\xi_j, \eta_j) $. Since $\sann(\coker(\varphi,\psi))=\sann(\coker(\psi,\varphi))$ with the abuse of notation we will also simplify the notation as $\sann(\coker(\varphi,\psi))=\sann(\varphi)$.

In general the annihilators of $D_n$ type have the following structure:
\begin{proposition}[General ideal structure of $D_n$]
Let $n$ be an odd integer. For $R = k[x, y]/f$ where $f = x^2 y + y^{n - 1}$ defines a $D_n$-type singularity, the stable annihilator ideal of the matrix factorizations $(\varphi_j, \psi_j)$ and $(\xi_j, \eta_j)$ for $j = 0, \dots, n - 3$ are given as $$\sann(\varphi_j) = \sann(\varphi_{n-j-2}) = (x^2, xy, y^{j+1}),\, 0 < j < \frac{n-1}{2},$$ and $$\sann(\xi_j) = \sann(\xi_{n-j-1}) = (x, y^j),\, 0 < j \leq \frac{n-1}{2},$$ respectively.
\end{proposition}
\begin{proof}
We start with $\varphi_j$. In order to show that $x^2,xy$ and $y^{j+1}$ are inside $\sann\varphi_j$, it is enough to observe that multiplication by these elements are nullhomotopic. We do so with the following equalities.
\begin{align*}
    x^2 I &=
\begin{bmatrix}
     x & y^j \\
     y^{n-j-2} & -x
 \end{bmatrix}
 \begin{bmatrix}
     x & y^j \\
     0 & -x
 \end{bmatrix}
 + \begin{bmatrix}
     x & y^j \\ 
     -y^{n-j-3} & 0 
 \end{bmatrix}
 \begin{bmatrix}
     xy & y^{j+1} \\
     y^{n-j-1} & -xy
 \end{bmatrix}
 \\
 xy I &=
\begin{bmatrix}
     x & y^j \\
     y^{n-j-2} & -x
 \end{bmatrix}
 \begin{bmatrix}
     -y^{j+1} & 0 \\
     xy & -y
 \end{bmatrix}
 + \begin{bmatrix}
     1 & 0 \\ 
     x & y^j 
 \end{bmatrix}
 \begin{bmatrix}
     xy & y^{j+1} \\
     y^{n-j-1} & -xy
 \end{bmatrix}
 \\
 y^{j+1} I &=
 \begin{bmatrix}
     x & y^j \\
     y^{n-j-2} & -x
 \end{bmatrix}
 \begin{bmatrix}
     -y & 0 \\
     y & -y
 \end{bmatrix}
 + \begin{bmatrix}
     1 & 0 \\
     1 & 1
 \end{bmatrix}
 \begin{bmatrix}
     xy & y^{j+1} \\
     y^{n-j-1} & -xy
 \end{bmatrix}
\end{align*}
So, we have $(x^2, xy, y^{j+1}) \subseteq \sann(\varphi_j)$. We will now show that there is equality. We claim that $x$ is not contained in $\varphi_j$. To see this, suppose to the contrary that it is the case. Thus, there are matrices 
\begin{align*}
    \begin{bmatrix}
        a & b \\ c & d 
    \end{bmatrix}
    \text{ and }
    \begin{bmatrix}
        e &  f\\ g & h 
    \end{bmatrix}
\end{align*}
such that
\begin{align*}
    xI = \begin{bmatrix}
     x & y^j \\
     y^{n-j-2} & -x
 \end{bmatrix}
 \begin{bmatrix}
     a & b \\ c & d
 \end{bmatrix}
 +
 \begin{bmatrix}
      e& f \\ 
      g & h
 \end{bmatrix}
 \begin{bmatrix}
     xy & y^{j+1} \\
     y^{n-j-1} & -xy
 \end{bmatrix}.
\end{align*}
This matrix equality gives us four equations.
\begin{align*}
    x &= ax + cy^j + exy + fy^{n-j-1}\\
    0 &= bx + dy^j + ey^{j+1} - fxy\\
    0 &= ay^{n-j-2}-cx+gxy+hy^{n-j-1}\\
    x &= by^{n-j-2} - dx + gy^{j+1} - hxy.
\end{align*}
The first equation says that $a=1$. Using it in the third one we get that $y^{n-j-2}-cx+gxy+hy^{n-j-1} = 0$. The only candidate for cancellation is the case when $h=1/y$ which can not happen. Hence we get that $x$ is not in any of the ideals of $\varphi_j$.

For $j=0$ and for all $n$, we have shown that $x^2,y\in \sann(\varphi_0)$ but $x\notin \sann(\varphi_0)$, hence we have $\sann(\varphi_0)=(x^2,y)$

Now we will show that $y^j\notin \sann(\varphi_j)$ for $0<j<\frac{n-1}{2}$. Similar to the above case, the following system of four equations does not have a solution.
\begin{align*}
    y^j &= ax + cy^j + exy + fy^{n-j-1}\\
    0 &= bx + dy^j + ey^{j+1} - fxy\\
    0 &= ay^{n-j-2}-cx+gxy+hy^{n-j-1}\\
    y^j &= by^{n-j-2} - dx + gy^{j+1} - hxy.
\end{align*}
Indeed, the first equation says $c=1$ but the third one can not hold due to the $-x$ term.

Hence we get that $\sann(\varphi_j)=(x^2,xy,y^{j+1})$ for $0<j<\frac{n-1}{2}$. Using the fact that $\coker(\varphi_j,\psi_j)\cong \coker(\varphi_{n-j-2},\psi_{n-j-2})$ the proof for the first set of ideals is complete.

We are now going to show computations for $\sann(\xi_j)$. We again start with three matrix equalities the first of which works for the case for $j = 0 $ (which is the case of a free module).

\begin{align*}
    I &=
\begin{bmatrix}
     x & 1 \\
     y^{n-1} & -xy
 \end{bmatrix}
 \begin{bmatrix}
     0 & 1 \\
     1 & -x
 \end{bmatrix}
 + \begin{bmatrix}
     x & 1 \\ 
     1 & 0 
 \end{bmatrix}
 \begin{bmatrix}
     xy & 1 \\
     y^{n-1} & -x
 \end{bmatrix} \\
 x I &=
\begin{bmatrix}
     x & y^j \\
     y^{n-j-1} & -xy
 \end{bmatrix}
 \begin{bmatrix}
     1 & y^j \\
     0 & -x
 \end{bmatrix}
 + \begin{bmatrix}
     x & y^j \\ 
     0 & -1 
 \end{bmatrix}
 \begin{bmatrix}
     xy & y^j \\
     y^{n-j-1} & -x
 \end{bmatrix}
 \\
 y^j I &=
\begin{bmatrix}
     x & y^j \\
     y^{n-j-1} & -xy
 \end{bmatrix}
 \begin{bmatrix}
     0 & y^j \\
     1 & -x
 \end{bmatrix}
 + \begin{bmatrix}
     x & y^j \\ 
     1 & 0 
 \end{bmatrix}
 \begin{bmatrix}
     xy & y^j \\
     y^{n-j-1} & -x
 \end{bmatrix}
\end{align*}

We will show that $y^{j-1}\notin \sann(\xi_j)$. Assuming to the contrary we get the following equation for the first entry
\begin{equation*}
    ax+cy^j +exy+fy^{n-j-1}=y^{j-1},
\end{equation*}
which implies that $f=y^{2j-n}$ but if $j\leq \frac{n-1}{2}$ this is not possible. Hence for $j\leq \frac{n-1}{2}$ we have
$\sann(\xi_j)=(x,y^j)$.

On the other hand, by \cite{Y}, $X_j \cong Y_{n - j - 1} = \coker(\psi_{n - j - 1}, \phi_{n - j - 1})$. Since $Y_{n - j - 1}$ is the first syzygy of $X_{n - j - 1}$, we have $\sann \eta_{n - j - 1} \cong \sann \xi_{n - j - 1}$. Thus, $\sann \xi_j = \sann \xi_{n - j - 1}$.
\end{proof}

\begin{example}
In case of $n=5$ we have
\begin{align*}
\begin{array}{lll}
\sann_R(A)=(x^2,y), & \sann_R(M_0)  =(x^2,y), & \sann_R(X_0)=R,\\
 & \sann_R(M_1)=(x^2,xy,y^2), & \sann_R(X_1)=(x,y),\\
 & \sann_R(M_2)=(x^2,xy,y^2), & \sann_R(X_2)=(x,y^2).\\
\end{array}
\end{align*}

Hence we get 
\begin{align*}
\cl(A)=\cl(M_0)=\{A, M_0,M_1,M_2\},\\
\cl(M_1)=\cl(M_2)=\{M_1,M_2\},\\
\cl(X_0)=\{A,M_0,M_1,M_2,X_0,X_1,X_2\},\\
\cl(X_1)= \{X_1,A,M_0,M_1,M_2\},\\
\cl(X_2)=\{X_2, M_1,M_2\}.
\end{align*}

Thus the closed sets of the Alexandrov topology is
\begin{align*}
    \{\emptyset, &\{M_1,M_2\},\{X_2,M_1,M_2\},\{A,M_0,M_1,M_2\},\{A,X_2,M_0,M_1,M_2\},\\
    &\{A,X_1,X_2,M_0,M_1,M_2\},
\{A,M_0,M_1,M_2,X_0,X_1,X_2\}\}.
\end{align*}

Enumerating the vertices as in the above order we get the following Hasse diagram for partial order.

\begin{figure}[h!]
    \centering

\tikzset{every picture/.style={line width=0.75pt}} 

\begin{tikzpicture}[x=0.75pt,y=0.75pt,yscale=-1,xscale=1]

\draw    (200,261.33) -- (200.33,310.5) ;
\draw   (192.42,317.25) .. controls (192.42,312.88) and (195.96,309.33) .. (200.33,309.33) .. controls (204.71,309.33) and (208.25,312.88) .. (208.25,317.25) .. controls (208.25,321.62) and (204.71,325.17) .. (200.33,325.17) .. controls (195.96,325.17) and (192.42,321.62) .. (192.42,317.25) -- cycle ;
\draw   (192.08,253.42) .. controls (192.08,249.04) and (195.63,245.5) .. (200,245.5) .. controls (204.37,245.5) and (207.92,249.04) .. (207.92,253.42) .. controls (207.92,257.79) and (204.37,261.33) .. (200,261.33) .. controls (195.63,261.33) and (192.08,257.79) .. (192.08,253.42) -- cycle ;
\draw    (154.5,208.75) -- (194.83,247.5) ;
\draw   (141.42,203.25) .. controls (141.42,198.88) and (144.96,195.33) .. (149.33,195.33) .. controls (153.71,195.33) and (157.25,198.88) .. (157.25,203.25) .. controls (157.25,207.62) and (153.71,211.17) .. (149.33,211.17) .. controls (144.96,211.17) and (141.42,207.62) .. (141.42,203.25) -- cycle ;
\draw    (204.83,247.25) -- (245,208.25) ;
\draw   (242.92,203.75) .. controls (242.92,199.38) and (246.46,195.83) .. (250.83,195.83) .. controls (255.21,195.83) and (258.75,199.38) .. (258.75,203.75) .. controls (258.75,208.12) and (255.21,211.67) .. (250.83,211.67) .. controls (246.46,211.67) and (242.92,208.12) .. (242.92,203.75) -- cycle ;
\draw    (153.33,196.75) -- (193.5,157.75) ;
\draw   (191.42,153.25) .. controls (191.42,148.88) and (194.96,145.33) .. (199.33,145.33) .. controls (203.71,145.33) and (207.25,148.88) .. (207.25,153.25) .. controls (207.25,157.62) and (203.71,161.17) .. (199.33,161.17) .. controls (194.96,161.17) and (191.42,157.62) .. (191.42,153.25) -- cycle ;
\draw    (206,158.25) -- (246.33,197) ;
\draw    (200,96.68) -- (200,145.08) ;
\draw   (192.08,89.52) .. controls (192.08,85.28) and (195.63,81.85) .. (200,81.85) .. controls (204.37,81.85) and (207.92,85.28) .. (207.92,89.52) .. controls (207.92,93.75) and (204.37,97.18) .. (200,97.18) .. controls (195.63,97.18) and (192.08,93.75) .. (192.08,89.52) -- cycle ;
\draw    (200.33,38) -- (200.33,82.58) ;
\draw   (192.42,30.08) .. controls (192.42,25.71) and (195.96,22.17) .. (200.33,22.17) .. controls (204.71,22.17) and (208.25,25.71) .. (208.25,30.08) .. controls (208.25,34.46) and (204.71,38) .. (200.33,38) .. controls (195.96,38) and (192.42,34.46) .. (192.42,30.08) -- cycle ;

\draw (195.33,309.5) node [anchor=north west][inner sep=0.75pt]   [align=left] {1};
\draw (194.33,245.5) node [anchor=north west][inner sep=0.75pt]   [align=left] {2};
\draw (143.5,196) node [anchor=north west][inner sep=0.75pt]   [align=left] {3};
\draw (245,196.5) node [anchor=north west][inner sep=0.75pt]   [align=left] {4};
\draw (194.33,145.67) node [anchor=north west][inner sep=0.75pt]   [align=left] {5};
\draw (194,82.5) node [anchor=north west][inner sep=0.75pt]   [align=left] {6};
\draw (194.33,23.5) node [anchor=north west][inner sep=0.75pt]   [align=left] {7};

\end{tikzpicture}

    \caption{Hasse diagram of $D_5$}
    \label{D5}
\end{figure}
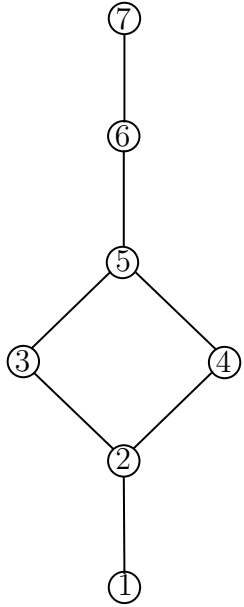
\end{example}

\begin{example}
In case of $n=7$ we have
\begin{align*}
\begin{array}{lll}
\sann_R(A)= (x^2, y), & \sann_R(M_0) = (x^2, y), & \sann_R(X_0)= R,\\
 & \sann_R(M_1) = (x^2, xy, y^2), & \sann_R(X_1)= (x, y),\\
 & \sann_R(M_2) = (x^2, xy, y^3), & \sann_R(X_2)= (x, y^2),\\
 & \sann_R(M_3) = (x^2, xy, y^3), & \sann_R(X_3)= (x, y^3),\\
 & \sann_R(M_4) = (x^2, xy, y^2), & \sann_R(X_4)= (x, y^2).\\
\end{array}
\end{align*}

Using the above calculations we find the closures as

\begin{gather*}
   \cl(M_2)=\cl(M_3)=\{M_2,M_3\},\\
   \cl(M_1)=\cl(M_4)= \{M_1,M_2,M_3,M_4\},\\
   \cl(M_0)=\cl(A)= \{A,M_0,M_1,M_2,M_3,M_4\},\\
    \cl(X_2)=\cl(X_4)=\{X_2,X_3,X_4,M_1,M_2,M_3,M_4\},\\  
    \cl(X_1) = \{A,M_0,M_1,M_2,M_3,M_4,X_1,X_2,X_3,X_4 \} ,\\
    \cl(X_0)= \{A,M_0,M_1,M_2,M_3,M_4,X_0, X_1,X_2,X_3,X_4 \}, \\
    \cl(X_3)=\{ X_3, M_2, M_3 \}.
    \end{gather*}

Thus, the closed sets of the Alexandrov topology is
\begin{gather*}
    \{ \emptyset,\cl(M_2), \cl(X_3), \cl(M_1), \cl(M_1)\bigcup \cl(X_3) , \cl(M_0), \cl(M_0)\bigcup \cl(X_3),\\ \cl(X_2),
\cl(X_2)\bigcup \cl(M_0), \cl(X_1), 
\cl(X_0) \}.
\end{gather*}

Enumerating the vertices as in the above order for closed sets, we get the following diagram:
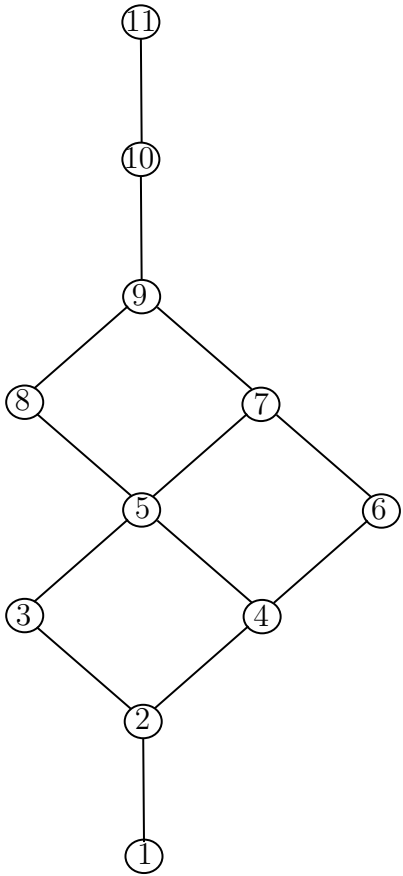
\begin{figure}[h!]
    \centering

\tikzset{every picture/.style={line width=0.75pt}} 

\begin{tikzpicture}[x=0.75pt,y=0.75pt,yscale=-1,xscale=1]

\draw    (291.53,384.75) -- (291.92,437.17) ;
\draw   (282.58,444.36) .. controls (282.58,439.7) and (286.76,435.92) .. (291.92,435.92) .. controls (297.08,435.92) and (301.26,439.7) .. (301.26,444.36) .. controls (301.26,449.02) and (297.08,452.8) .. (291.92,452.8) .. controls (286.76,452.8) and (282.58,449.02) .. (282.58,444.36) -- cycle ;
\draw   (282.19,376.31) .. controls (282.19,371.65) and (286.37,367.87) .. (291.53,367.87) .. controls (296.69,367.87) and (300.87,371.65) .. (300.87,376.31) .. controls (300.87,380.97) and (296.69,384.75) .. (291.53,384.75) .. controls (286.37,384.75) and (282.19,380.97) .. (282.19,376.31) -- cycle ;
\draw    (237.85,328.7) -- (285.43,370) ;
\draw   (222.42,322.83) .. controls (222.42,318.17) and (226.6,314.39) .. (231.76,314.39) .. controls (236.91,314.39) and (241.1,318.17) .. (241.1,322.83) .. controls (241.1,327.49) and (236.91,331.27) .. (231.76,331.27) .. controls (226.6,331.27) and (222.42,327.49) .. (222.42,322.83) -- cycle ;
\draw    (297.23,369.74) -- (344.61,328.16) ;
\draw   (342.16,323.36) .. controls (342.16,318.7) and (346.34,314.93) .. (351.5,314.93) .. controls (356.65,314.93) and (360.84,318.7) .. (360.84,323.36) .. controls (360.84,328.03) and (356.65,331.8) .. (351.5,331.8) .. controls (346.34,331.8) and (342.16,328.03) .. (342.16,323.36) -- cycle ;
\draw    (236.47,315.9) -- (283.86,274.33) ;
\draw   (281.4,269.53) .. controls (281.4,264.87) and (285.58,261.09) .. (290.74,261.09) .. controls (295.9,261.09) and (300.08,264.87) .. (300.08,269.53) .. controls (300.08,274.19) and (295.9,277.97) .. (290.74,277.97) .. controls (285.58,277.97) and (281.4,274.19) .. (281.4,269.53) -- cycle ;
\draw    (298.61,274.86) -- (346.19,316.17) ;
\draw    (357.39,316.44) -- (404.78,274.86) ;
\draw   (402.32,270.06) .. controls (402.32,265.4) and (406.5,261.62) .. (411.66,261.62) .. controls (416.82,261.62) and (421,265.4) .. (421,270.06) .. controls (421,274.72) and (416.82,278.5) .. (411.66,278.5) .. controls (406.5,278.5) and (402.32,274.72) .. (402.32,270.06) -- cycle ;
\draw    (296.64,262.6) -- (344.02,221.02) ;
\draw   (341.57,216.23) .. controls (341.57,211.57) and (345.75,207.79) .. (350.91,207.79) .. controls (356.06,207.79) and (360.25,211.57) .. (360.25,216.23) .. controls (360.25,220.89) and (356.06,224.67) .. (350.91,224.67) .. controls (345.75,224.67) and (341.57,220.89) .. (341.57,216.23) -- cycle ;
\draw    (358.77,221.56) -- (406.35,262.87) ;
\draw    (237.85,221.02) -- (285.43,262.33) ;
\draw   (222.42,215.16) .. controls (222.42,210.5) and (226.6,206.72) .. (231.76,206.72) .. controls (236.91,206.72) and (241.1,210.5) .. (241.1,215.16) .. controls (241.1,219.82) and (236.91,223.6) .. (231.76,223.6) .. controls (226.6,223.6) and (222.42,219.82) .. (222.42,215.16) -- cycle ;
\draw    (236.47,208.23) -- (283.86,166.66) ;
\draw   (281.4,161.86) .. controls (281.4,157.2) and (285.58,153.42) .. (290.74,153.42) .. controls (295.9,153.42) and (300.08,157.2) .. (300.08,161.86) .. controls (300.08,166.52) and (295.9,170.3) .. (290.74,170.3) .. controls (285.58,170.3) and (281.4,166.52) .. (281.4,161.86) -- cycle ;
\draw    (298.61,167.19) -- (346.19,208.5) ;
\draw    (290.35,101.01) -- (290.74,153.42) ;
\draw   (281.01,92.57) .. controls (281.01,87.9) and (285.19,84.13) .. (290.35,84.13) .. controls (295.51,84.13) and (299.69,87.9) .. (299.69,92.57) .. controls (299.69,97.23) and (295.51,101.01) .. (290.35,101.01) .. controls (285.19,101.01) and (281.01,97.23) .. (281.01,92.57) -- cycle ;
\draw    (290.35,31.71) -- (290.74,84.13) ;
\draw   (281.01,23.27) .. controls (281.01,18.61) and (285.19,14.83) .. (290.35,14.83) .. controls (295.51,14.83) and (299.69,18.61) .. (299.69,23.27) .. controls (299.69,27.93) and (295.51,31.71) .. (290.35,31.71) .. controls (285.19,31.71) and (281.01,27.93) .. (281.01,23.27) -- cycle ;

\draw (287.01,436.66) node [anchor=north west][inner sep=0.75pt]   [align=left] {1};
\draw (285.83,368.43) node [anchor=north west][inner sep=0.75pt]   [align=left] {2};
\draw (225.86,315.66) node [anchor=north west][inner sep=0.75pt]   [align=left] {3};
\draw (345.6,316.2) node [anchor=north west][inner sep=0.75pt]   [align=left] {4};
\draw (285.83,262.01) node [anchor=north west][inner sep=0.75pt]   [align=left] {5};
\draw (404.98,262.36) node [anchor=north west][inner sep=0.75pt]   [align=left] {6};
\draw (345.6,209.24) node [anchor=north west][inner sep=0.75pt]   [align=left] {7};
\draw (225.27,207.46) node [anchor=north west][inner sep=0.75pt]   [align=left] {8};
\draw (284.26,154.16) node [anchor=north west][inner sep=0.75pt]   [align=left] {9};
\draw (280,85) node [anchor=north west][inner sep=0.75pt]   [align=left] {10};
\draw (281,16) node [anchor=north west][inner sep=0.75pt]   [align=left] {11};

\end{tikzpicture}

\caption{Hasse diagram of $D_7$}
\end{figure}
\end{example}
\end{chunk}

\newpage

\begin{example}
In section 4 we had claimed that $M \in \cl_n(N) \nRightarrow \cl_n(M) \subseteq \cl_n(N)$, now we give an example. Consider $D_7$ with the operator $\cl_2$ and the ideals $N=\sann(\varphi_3)=(x^2,xy,y^3)$, $M=\sann(\varphi_1)=(x^2,xy,y^2)$, and $L=\sann(\xi_1)=(x,y)$. We see that $M\in \cl_2(N)$, $L\in \cl_2(M)$ but $L\notin \cl_2(N)$ so we have $M\in \cl_2(N)$ but $\cl_2(M)\not\subseteq \cl_2(N)$.
\end{example}

\begin{chunk}{Type E.}

We first consider the case $E_6$. Let $R=k[x,y]/(f)$ with $f=x^3+y^4$. Its matrix factorizations are given as $(\varphi_1,\psi_1),(\varphi_2,\psi_2)$ and $(\alpha,\beta)$ as in \cite{Y}. Calculating the stable annihilators of the cokernels we get
\begin{align*}
    \sann_R(\varphi_1)&=\sann_R(\psi_1)=(x,y),\\
    \sann_R(\varphi_2)&=\sann_R(\psi_2)=(x,y^2),\\
    \sann_R(\alpha)&=\sann_R(\beta)=(x^2,y^2,xy).
\end{align*}
With the abuse of notation we replace $\coker(\varphi_i)$ with $\varphi_i$, same goes for all matrices, hence we have
\begin{align*}
    \cl(\varphi_1)= X , \;\; \cl(\varphi_2)=\{\alpha,\varphi_2\} ,\;\;\cl(\alpha) = \{\alpha\} .
\end{align*}

We see that closed set are again totally ordered, meaning that we have a straight Hasse diagram for $E_6$.

For $E_7$ we let $R=k[x,y]/(x^3+xy^3)$. With the same procedure as above we get
\begin{align*}
    \sann_R(\alpha)&= (x,y^3), & \sann_R(\gamma)&= (x^2,xy,y^3), & \sann_R(\varphi_1) &=\sann_R(\eta_3)=(x,y), \\ \sann_R(\varphi_2)&= (x,y^2),  & \sann_R(\eta_1)&= (x^2,xy,y^2), & \sann_R(\eta_2)&= (x^2,xy,y^3).
\end{align*}

From which we calculate the closures as
\begin{align*}
    \cl(\alpha)&= \{\alpha,\gamma,\eta_2\}, & \cl(\eta_2)&=\cl(\gamma)=\{\gamma,\eta_2\}, & \cl(\varphi_1)=\cl(\eta_3)&=X, \\ \cl(\varphi_2)&=\{\alpha,\gamma,\varphi_2,\eta_1,\eta_2\}, & \cl(\eta_1)&=\{\gamma,\eta_1,\eta_2\}.
\end{align*}

Hence, the topology given by the closed sets is $$\tau=\{\emptyset,\{\gamma,\eta_2\},\{\alpha,\gamma,\eta_2\},\{\gamma,\eta_1,\eta_2\},\{\alpha,\gamma,\eta_1,\eta_2\},\{\alpha,\gamma,\varphi_2,\eta_1,\eta_2\},X\}\;.$$
This topology has a similar Hasse diagram to the Hasse diagram of $D_5$ (\ref{D5}).

For the case of $E_8$ we consider $k[x,y]/(x^3+y^5)$. Calculating the stable annihilators of the cokernels of the indecomposable matrix factorizations we get
\begin{align*}
    \sann_R(\varphi_1)=& (x,y), \;\;\; \sann_R(\varphi_2)= (x,y^2), \;\;\; \sann_R(\gamma_1)=\sann_R(\alpha_1) =(x^2,xy,y^2), \\ \sann_R&(\alpha_2)=\sann_R(\gamma_2)=\sann_R(\xi_1)=\sann_R(\xi_2) = (x^2,xy,y^3)\;\;.
\end{align*}

As these ideals are totally ordered by inclusion we get that the topology on $E_8$ is given by the straight Hasse diagram. This concludes the examples of $E_6,E_7$ and $E_8$ type singularities.

\end{chunk}

\bibliographystyle{alpha}
\bibliography{alexandrov}
\end{document}